\documentclass[12pt]{amsart}
\usepackage[utf8]{inputenc}
\usepackage[total={6.5in,9in},top=1in, left=1in, right=1in, bottom=1in]{geometry}
\usepackage{fancyhdr}
\usepackage{libertine}
\usepackage{color,amsmath,amssymb,mathtools,graphicx}
\usepackage{epsfig,fancyhdr,xcolor}
\usepackage{dsfont} 
\usepackage{float}
\usepackage{epsfig,color,fancyhdr,setspace}
\usepackage{import}
\usepackage{lineno}
\usepackage{stackrel}
\usepackage[allcolors = blue,colorlinks]{hyperref}
\usepackage[abs]{overpic}
\usepackage[center]{caption}
\usepackage{subcaption}
\usepackage{tikz-cd}
\usepackage{enumitem}
\usepackage{wasysym}
\usepackage[normalem]{ulem}

\newtheorem{theorem}{Theorem}[section]

\newtheorem{definition}[theorem]{Definition}

\newtheorem{proposition}[theorem]{Proposition}
\newtheorem{remark}[theorem]{Remark}
\newtheorem{corollary}[theorem]{Corollary}
\newtheorem{conjecture}[theorem]{Conjecture}

\title{On framings of links in 3-manifolds}

\author{Rhea Palak Bakshi}
\address{Department of Mathematics, The George Washington University, Washington DC, USA}
\email{rhea\_palak@gwu.edu}
\author{Dionne Ibarra}
\address{Department of Mathematics, The George Washington University, Washington DC, USA}
\email{dfkunkel@gwu.edu}
\author{Gabriel Montoya-Vega}
\address{Department of Mathematics, The George Washington University, Washington DC, USA}
\email{gmontoyavega@gwu.edu}
\author{J\'{o}zef H. Przytycki}
\address{Department of Mathematics, The George Washington University, Washington DC, USA and \linebreak Department of Mathematics, University of Gda\'{n}sk, Poland}
\email{przytyck@gwu.edu}
\author{Deborah Weeks}
\address{Department of Mathematics, The George Washington University, Washington DC, USA}
\email{deweeks@gwu.edu}

\keywords{knots, links, $3$-manifolds, framings of links, skein modules, spin structures, Dehn homeomorphisms, incompressible surfaces.}
\subjclass[2010]{Primary: 57M27. Secondary: 57M25}

\begin{document}

\begin{abstract}

We show that the only way of changing the framing of a link by ambient isotopy in an oriented $3$-manifold is when the manifold has a properly embedded non-separating $S^2$. This change of framing is given by the Dirac trick, 
also known as the light bulb trick. The main tool we use is based on McCullough's work on the mapping class groups of $3$-manifolds. We also relate our results to the theory of skein modules.

\end{abstract} 

\maketitle

\tableofcontents

\section{Introduction}

We show that the only way to change the framing of a link in an  oriented $3$-manifold by ambient isotopy 
is when the manifold has a properly embedded non-separating $S^2$. More
precisely the only change of framing is by the light bulb trick as illustrated in Figures \ref{lbl} and \ref{linklbl}. 
Here the change of framing is very local (takes part in $S^2\times [0,1]$ embedded in the manifold) and is related to the fact
that the fundamental group of $SO(3)$ is $\mathbb{Z}_2$. Furthermore, we use the fact that $3$-manifolds possess 
spin structures given by the parallelization of their tangent bundles to show that the framing can only be changed by an even number of full twists. 
We give a short outline of the history of the problem in Subsection \ref{history}.

\begin{figure}[h]
    \centering
    \includegraphics[scale = 0.45]{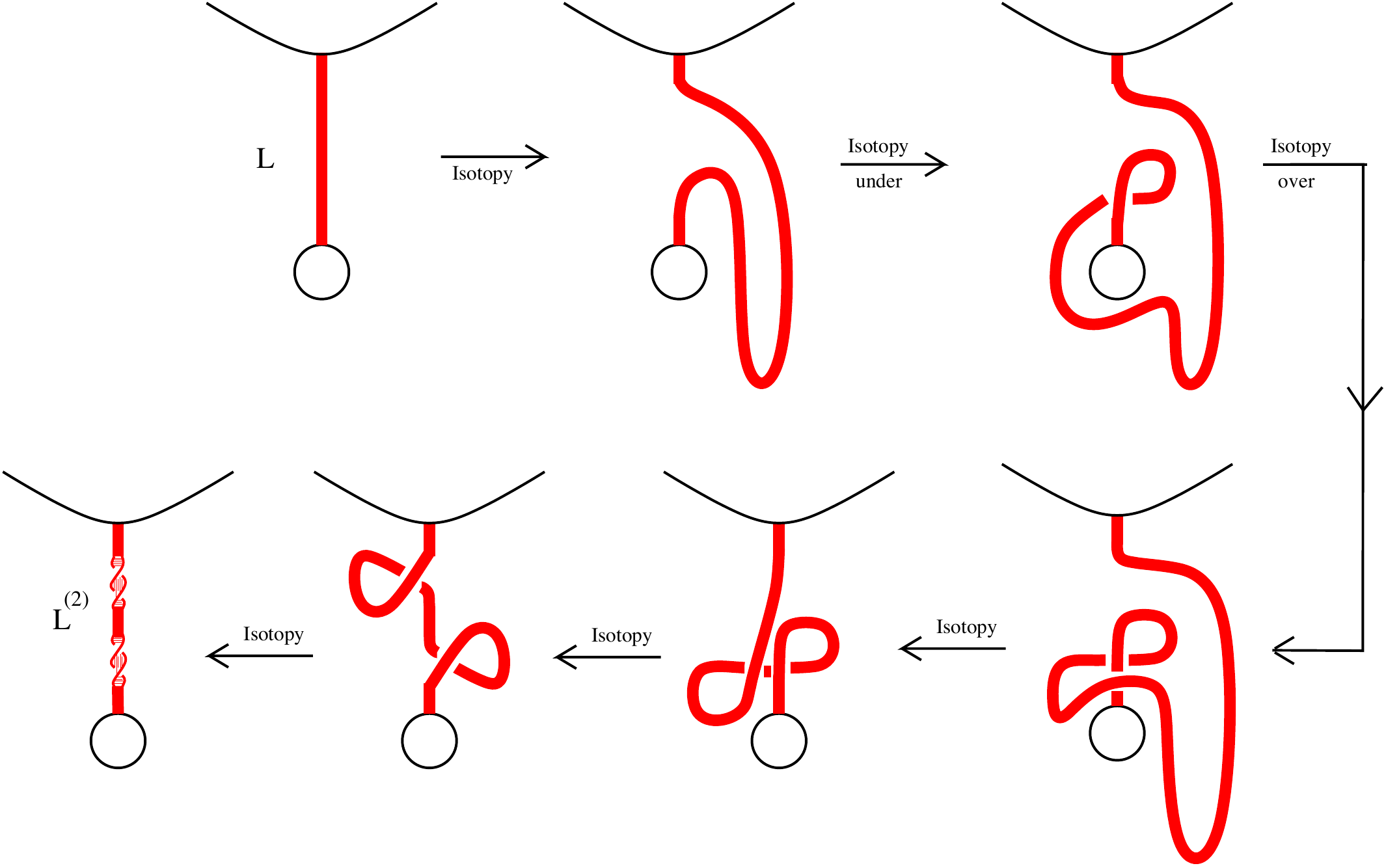}
    \caption{Dirac trick for a knot illustrated using a light bulb}
    \label{lbl}
\end{figure}

In Section 2 we introduce the following preliminary material from $3$-dimensional topology: incompressible surfaces,
mapping class groups of $3$-manifolds, and Dehn homeomorphisms.
In Section 3 we prove our main result (see Theorem \ref{Link}). Furthermore, we extend our main result to non-compact $3$-manifolds. Our main tools are various results of D. McCullough about mapping class groups of $3$-manifolds. The use of the Papakyriakopoulos' Loop theorem \cite{papaloop} and its generalization to the annulus, proved by Waldhausen in \cite {genloop}, is in the background of the proof as it was used by McCullough to prove Theorem \ref{mccullough2006-1}. 
Additionally, in Subsection \ref{spinandframing} we use spin structures on $3$-manifolds to show that the framing of 
a knot or a link cannot be changed by an odd number of full twists. For an introduction to spin structures from the point of view used in this paper, we refer the reader to \cite{ssandspin} and \cite{spinstructure}.

In Section 4 we describe the  ramifications of our results in the context of skein modules, for example, the framing skein module,
$q$-homology skein module, and the Kauffman bracket skein module. In particular, we compute the exact structure of  the framing skein module of links in an oriented $3$-manifold.  

In Section 5 we discuss future directions of research, in particular, questions about incompressible surfaces and 
torsion in various skein modules. Futhermore, we conjecture that an oriented irreducible non-Haken $3$-manifold cannot
have torsion in its Kauffman bracket skein module. We also relate our results to Witten's conjecture on the finiteness of Kauffman bracket skein modules.

\subsection{History of the problem}\label{history}
When the fourth author was visiting Michigan State University  in East Lansing in the Fall of 1989 and was working on the Kauffman bracket skein module
of lens spaces, he thought about the possibility of changing framings of knots in $3$-manifolds by ambient isotopy.
He found an argument that such changes of framing are impossible for irreducible $3$-manifolds. Y. Rong, then a postdoc at MSU, pointed out (see \cite{homotopysm}) that from McCullough's work it follows that if $M$ is a compact $3$-manifold then the framing of a knot in $M$ cannot be changed by ambient isotopy if and only if $M$ has no non-separating $S^2$. The fourth author was then referred to \cite{mapred} and \cite{homeobook} by D. McCullough from which it follows that the only possibility of changing the framing of a knot in $M$ is by the Dirac trick. 
The fourth author stated this result several times (for example, see \cite{qanalogue}) but without giving a detailed proof. Thus, when Vladimir Chernov became interested in the problem he gave his own proof in \cite{framedknots}\footnote{In this paper Chernov wrote, {\it ``This result (for compact manifolds) was first stated by Hoste and Przytycki \cite{homotopysm}. They referred to the work \cite{mapred} of McCullough on mapping class groups of $3$-manifolds for
the idea of the proof of this fact. However to the best of our knowledge the proof
was not given in the literature. The proof we provide is based on the ideas and methods different from the ones Hoste and Przytycki had in mind."}}. When the authors of this paper decided to give details of the old proof they noted that from the new paper \cite{mcgthennow} the proof could be deduced in a relatively simple
way (but still using the generalized Loop theorem for the annulus). Soon after, the authors found another paper by D. McCullough \cite{homeodehnboundary} 
from which the main result follows more directly and this is the proof presented in this paper. After a preprint of our paper
was posted on arXiv, we were kindly informed by Chernov about the paper \cite{CCS} where they prove the result for knots using McCullough's results. Thus, the novelty of this paper is the result for links and, in particular, its relation to skein modules.

\section{Preliminaries}

The following definitions and theorems in low-dimensional topology are useful in the proofs of our results.

\begin{definition}
Let $F$ be a properly embedded surface in a $3$-manifold $M$ or be a part of the boundary of $M$. A disk $D^2$ in $M$ is called a {\pmb compressing disk} of $F$ if $D^2 \cap F = \partial D^2$  and $\partial D^2$ is not contractible in $F$.

An embedded surface $F$ in $M$ that admits a compressing disk is said to be \textbf{compressible}. If the surface $F$ is not a disjoint union of $S^2$'s or $D^2$'s and it contains no compressing disk then the surface is said to be \textbf{incompressible}. If $F = S^2$, then $F$ is incompressible if it does not bound a $3$-ball in $M$.\\

\end{definition}

\begin{definition}
Let $M$ be an orientable manifold and $Homeo(M)$ denote the space of PL orientation preserving
homeomorphisms of $M$. Then the \textbf{mapping class group} of $M$, denoted by $\mathcal{H}(M)$, 
is defined as the space of all ambient isotopy classes of $Homeo(M)$. \\

\end{definition}

\begin{definition}

Let $(F^{n-1} \times I, \partial F^{n-1} \times I) \subset 
(M^n, \partial M^n)$, where $F$ is a connected codimension $1$ submanifold, and $(F \times I) \cap \partial M = \partial F \times I$. Let $\langle \phi_t\rangle$ be an element of $\pi_1(Homeo(F), 1_F )$, that is, for $0 \leq t \leq 1$, $\phi_t$ is a continuous family of homeomorphisms of $F$ such that $\phi_0 = \phi_1 = 1_F$. Define a \textbf{Dehn homeomorphism} as $h \in \mathcal{H}(M)$ by:

\[
  h =
  \begin{cases}
h(x, t) = (\phi_t(x), t) & \text{if $(x, t) \in F \times I$} \\

h(m) = m & \text{if $m \notin F \times I$} \\
 
  \end{cases}
\]

Dehn homeomorphisms are generalizations of Dehn twist homeomorphisms of a surface. 
When $\pi_1(Homeo(F))$ is trivial, a Dehn homeomorphism must be isotopic to the identity. 
Define the \textbf{Dehn subgroup} $\mathcal{D}(M)$ of $\mathcal{H}(M)$ to be the subgroup generated by all Dehn
homeomorphisms.

\end{definition}

\begin{theorem}[Finite Mapping Class Group Theorem]\cite{mcgthennow}\label{finitemcgnonhaken} Let $M$ be a closed oriented irreducible non-Haken $3$-manifold, that is, $M$ is irreducible and contains no properly embedded, incompressible, two-sided surface. Then $\mathcal{H}(M)$ is finite.

\end{theorem}

This theorem is also true for all hyperbolic $3$-manifolds. 

\begin{theorem}\label{hyperbolicfinmcg}\cite{gmt} The mapping class group of a compact hyperbolic $3$-manifold is finite.\\

\end{theorem}

\begin{theorem}\cite{homeodehnboundary}\label{mccullough2006-1}
Let $M$ be a compact orientable $3$-manifold which admits a homeomorphism which is Dehn twists on the boundary
about the collection $C_1,\hdots,C_n$ of simple closed curves in $\partial M$. Then for each $i$, either $C_i$ bounds a disk
in $M$, or for some $j\neq i$, $C_i$ and $C_j$ cobound an incompressible annulus in $M$. \\

\end{theorem}

\begin{corollary}\label{cortomccullough2006-1}

Let $M'$ be a compact oriented $3$-manifold with one of the boundary components of $M'$ being a torus. 
	We denote this torus by $\partial _0 M'$. Let $\widetilde f: M' \longrightarrow M'$ 
	be a homeomorphism which acts nontrivially on $\partial _0 M'$ and is constant on $\partial M' \setminus \partial _0 M'$. 
	Then $\partial _0 M'$ has a compressing disk.
 
\end{corollary}

\begin{proof}

$\widetilde f|_{\partial _0 M'}$ is generated by Dehn twists on a family $\mathcal{C}$ of $k\ (\neq 0)$ parallel 
	nontrivial simple closed curves on the torus $\partial _0 M'$. Therefore, by Theorem \ref{mccullough2006-1}  these curves bound 
	(compressing) disks or incompressible annuli. However, Dehn twists along an annulus 
	with both boundary components on $\partial _0 M'$ act trivially on $\partial _0 M'$. 
	Therefore, each curve in $\mathcal{C}$ bounds a compressing disk.

\end{proof}

\begin{theorem}\cite{homeodehnboundary}\label{mccullough2006-2}
Let $M$ be a compact orientable $3$-manifold which admits a homeomorphism which is Dehn twists on the boundary
about the collection $C_1,...,C_n$ of simple closed curves in $\partial M$. Then there is a collection of disjoint imbedded disks  and annuli in $M$, each of whose boundary circles is isotopic to one of the $C_i$, for which some  
composition of Dehn twists about these disks and annuli is isotopic to $h$ on $\partial M$. 

\end{theorem}

That is, $h$ must arise in the most obvious way, by composition of Dehn twists about a collection of disjoint annuli
and disks  with a homeomorphism that is the identity on the boundary.

\begin{corollary}\label{manytori}
        
Let $M'$ be a compact oriented $3$-manifold with some boundary components, say, \\ $\partial_1(M'),\hdots,\partial_k(M')$ being
tori. Let $\widetilde f:M' \longrightarrow M'$ be a homeomorphism which acts nontrivially on every $\partial_i(M')$ and
 which is the identity on $\partial M' \setminus \bigcup_i\partial_i(M')$. Then either $\partial_i(M')$ has a compressing disk,
say $D_i^2$, or there is some $j\neq i$ such that there is an incompressible annulus $Ann_{i,j}$ with one boundary component on $\partial_i(M')$ and the other on $\partial_j(M')$. Furthermore, one can take these disks and annuli to be disjoint. 
Also,  $\widetilde f$ restricted to $\partial M'$ is ambient isotopic to the composition (of some powers) of the Dehn homeomorphisms along $D_i^2$ and $Ann_{i,j}$.

\end{corollary}

For the remainder of the paper we will denote a framed knot (respectively link) by {\emph {\pmb K}} (respectively {\pmb L)} and the underlying unframed knot (respectively link) by $K$ (respectively $L$).

\section{Main Results}

\begin{theorem}\label{Link}

Let $\mathbf L$ be a framed link in a compact oriented $3$-manifold $M$. The only way of changing the framing of \ ${\pmb L}$ by ambient isotopy while preserving the components of $L$ is when $M$  has a properly embedded non-separating $2$-sphere and either:

\begin{enumerate}

  \item  [(i)] $L$ intersects the non-separating $2$-sphere transversely exactly once, or 
\item [(ii)] $L$ intersects the non-separating $2$-sphere transversely in two points, each point belonging to a different component of $L$.

\end{enumerate}

The framing is changed by a composition of even powers of the Dehn homeomorphisms along the disjoint union of $S_j^2$, where $S_j^2$ satisfy conditions (i) or (ii). We illustrate the change of framing by the light bulb trick \footnote{For a nice computer animation of the Dirac trick we refer the reader to \cite{animation}.} (in the regular neighborhood of $S^2$) in Figure \ref{lbl} if there is one point of intersection and in Figure \ref{linklbl} if there are exactly two points of intersection of $L$ with $S^2$.

\end{theorem}

\begin{figure}[h]
    \centering
    \includegraphics[scale = 0.40]{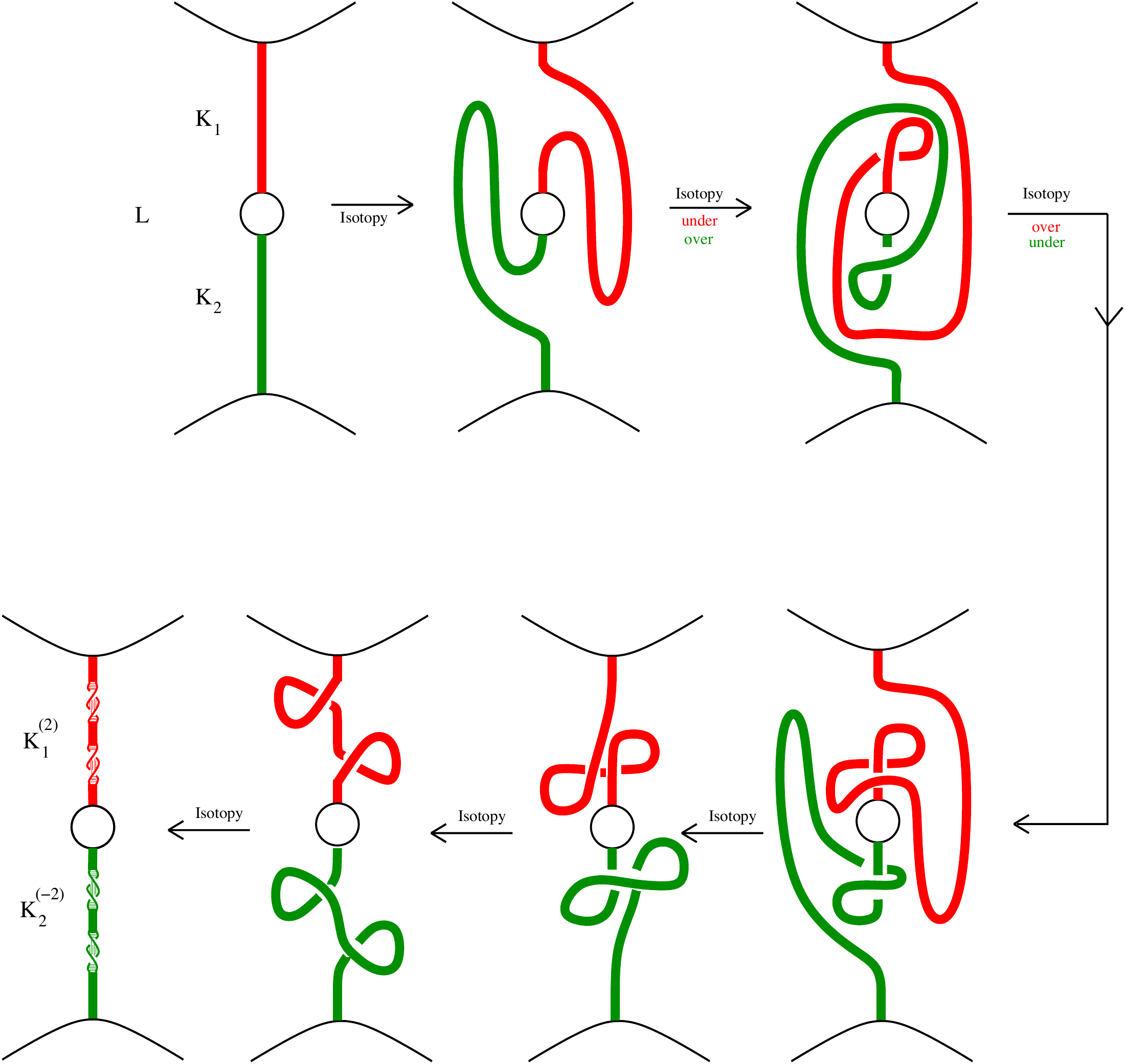}
     \caption{Dirac trick for a link with two components illustrated using a light bulb}
    \label{linklbl}
\end{figure}

The following theorem is the key result used in proving Theorem \ref{Link}  and is of considerable interest in itself.

\begin{theorem}\label{nonseparating}

Let ${\pmb L}$ be a framed link in a compact oriented $3$-manifold $M$. If $f:M \longrightarrow M$ is a homeomorphism 
	which changes the framing of \ ${\pmb L}$ while preserving the components of $L$ and $f|_{\partial M} = Id$, then
there is a non-separating $S^2$ in $M$ and either:

\begin{enumerate}

  \item  [(i)] $L$ intersects the non-separating $2$-sphere transversely exactly once, or 
\item [(ii)] $L$ intersects the non-separating $2$-sphere transversely in two points, each point belonging to a different component of $L$.

\end{enumerate}

The framing is changed by a composition of powers of the Dehn homeomorphisms along the disjoint union of $S_j^2$, where $S_j^2$ satisfy conditions (i) or (ii).

\end{theorem}

\begin{proof}

Let the knots ${\pmb K}_1,{\pmb K}_2,\hdots,{\pmb K}_k$ be the components of the framed link $\mathbf L$, 
 $f:M \longrightarrow M$ be a homeomorphism which changes the framing of $\pmb L$ while preserving the components of $L$, and $f|_{\partial M} = Id$. Without loss of generality, we can assume that $f$ changes the framing of all the components of $\pmb L$. Additionally, we can assume that there are regular neighborhoods $V_{K_i}$ of $K_i$ so that $f(V_{K_i}) =V_{K_i}$ (see \cite{hudsonpl}).
Let $M'= M \setminus \bigcup_i int(V_{K_i})$ and $\widetilde f= f|_{M'}: M' \longrightarrow M'$.
Thus, $\widetilde f$ acts nontrivially on all $\partial V_{K_i}$. Therefore, $\mu_i \mapsto \mu_i$ and $\lambda_i \mapsto \lambda_i + \ell_i\mu_i$, where $\mu_i$ 
is the meridian of $V_{K_i}$ and $\lambda_i$ is a fixed longitude, that is, it is a curve  which intersects $\mu_i$ exactly once. Here $\ell_i \neq 0$ and $f(\pmb K_i) = \pmb K_i^{\ell_i}$. By Corollary \ref{cortomccullough2006-1}, for any $i$ there is a compressing
disk $D_i^2$ of $\partial V_{K_i}$, or there exists $j$ ($i\neq j$) and an incompressible annulus $Ann_{i,j}$ with one boundary component on $\partial V_{K_i}$ and the other on $\partial V_{K_j}$. By Corollary \ref{manytori} we can assume that these compressing disks and incompressible annuli are disjoint. Now the Dehn twist along $\partial D_i^2 $ must be the same as the Dehn twist along $\partial D_{\mu_i}^2$, otherwise $\widetilde f$ cannot be extended to $f$. Therefore, the boundary of the compressing disk $D_i^2$ is the same as the boundary of the meridian disk $D_{\mu_i}^2$ of $V_i$ and $D_i^2\cup_{\partial} D_{\mu_i}^2=S^2_i$ is a non-separating $2$-sphere intersecting $K_i$ transversely in one point. Similarly, $\partial Ann_{i,j}= \partial D_{\mu_i}^2 \sqcup \partial D_{\mu_j}^2$
and $D_{\mu_i}^2 \cup Ann_{i,j} \cup D_{\mu_j}^2 = S_{i,j}^2$ is a non-separating $2$-sphere in $M$ cut by $K_i$ and $K_j$ transversely in one point each. Furthermore, we can assume
that the action of $\widetilde f$ on $\partial V_{K_i}$  is the same as that given by some composition of Dehn homeomorphisms along disjoint compressing disks $D_i^2$ and incompressible annuli $Ann_{i,j}$.

\end{proof}

\begin{remark}

For oriented $3$-manifolds $M$ for which $\mathcal{H}(M')$ is finite, Theorems \ref{Link} and \ref{nonseparating} can be deduced immediately. 
This is because $f|_{\partial V_{K_i}}$ has infinite order in $\mathcal{H}(\partial V_{K_i})$, and thus, $\widetilde {f}$ 
has infinite order in $\mathcal{H}(M')$. Hence, we get a contradiction. This happens, in particular, when $M'$ 
is a hyperbolic manifold or a closed oriented irreducible non-Haken $3$-manifold (see  
	Theorems \ref{finitemcgnonhaken} and \ref{hyperbolicfinmcg}).

\end{remark}

\subsection{Proof of the Main Theorem}

\begin{proof}

We have shown in Theorem \ref{nonseparating} that the only way to change the framing of ${\pmb L}$ is by Dehn homeomorphisms 
along non-separating $2$-spheres which are either intersected by $L$ exactly once or they are intersected by $L$ at two points, each of which belongs to a different component of $L$. Let $\tau_i$ be the Dehn homeomorphism (rotation) about the non-separating sphere $S_i^2$ and $\tau_{i,j}$ the Dehn homeomorphism about the non-separating sphere $S_{i,j}^2$. Now since $\pi_1(SO_3) = \mathbb{Z}_2$, 
	we get that $\tau_i ^2$ and $\tau_{i,j}^2$ are ambient isotopic to the identity map. Also, $\tau_i^2$ and $\tau_{i,j}^2$ along the non-separating $2$-spheres realize the light bulb trick. Furthermore, unlike in Theorem \ref{nonseparating}, the function $f$ which changes 
	the framing of ${\pmb L}$ is ambient isotopic to the identity map. From the existence of spin 
	structures on every oriented $3$-manifold it follows that this function can be considered to be an even 
	power of the rotations $\tau_i$ and $\tau_{i,j}$ along $S_i^2$ and $S_{i,j}^2$, respectively (see Subsection 
	\ref{spinandframing}, \cite{ssandspin}, and \cite{spinstructure}). This completes the proof of Theorem \ref{Link}.\\
	 The fact that $\tau_i^2$ twists the framing of ${\pmb K_i}$ twice by ambient isotopy is illustrated in Figure \ref{lbl}, and that $\tau_{i,j}^2$ simultaneously twists the framing of ${\pmb K_i}$ twice and of ${\pmb K_j}$ twice but in the opposite direction, is illustrated in Figure \ref{linklbl}.\\

\end{proof}

\begin{corollary}\label{noncompact}

Theorem \ref{Link} also holds for non-compact oriented $3$-manifolds.

\end{corollary}

\begin{proof}
The ambient isotopy of $M$, which changes the framing of ${\pmb L}$ can be taken to have support on a finite number
of $3$-balls in $M$ \footnote{It follows from Theorem 6.2 in \cite{hudsonpl} that if $C$ is a compact subset of a manifold
$M$ and $F:M\times I \longrightarrow M$ is an ambient isotopy of $M$ then there is another
ambient isotopy $\hat F:M\times I \longrightarrow M$ such that $F_0 ={\hat F}_0$, $F_1\setminus C = {\hat F}_1\setminus C$ and there exists a number $N$ such that the set $\{x\in M \ |\ (\hat F\setminus\{x\}) \times (k/N,(k+1)/N) \ is \ not \ constant\}$ sits in a ball embedded in $M$.}.
Thus, the new ambient isotopy has a compact oriented $3$-manifold as a support and the result follows from Theorem \ref{Link}.

\end{proof}

The following corollary to Theorem \ref{Link} was stated and proven in \cite{homotopysm}, \cite{framedknots}, and \cite{CCS}. 

\begin{corollary}\label{mainresult}

Let ${\pmb K}$ be a framed knot in an oriented $3$-manifold $M$. The only way of changing the framing of \ ${\pmb K}$ by an ambient isotopy of $M$ is when the manifold has a properly embedded non-separating $S^2$ and the underlying knot $K$ intersects this $S^2$ transversely exactly once. 
More precisely, the only change of framing is by the light bulb trick. Equivalently, all possible changes of 
the framing of \ ${\pmb K}$ can be realized by even powers of Dehn homeomorphisms along a non-separating $S^2$ 
which is cut by $K$ exactly once.

\end{corollary}

\begin{remark}\hfill
	\begin{enumerate}
		\item[(a)]
If $M$ is an integral homology sphere (respectively, rational homology sphere), then every knot in $M$ has a preferred 
	framing (respectively, rational framing). As mentioned before, for arbitrary oriented $3$-manifolds we can only define modulo $2$ framing which reflects the affine space of 
	spin structures over $H^1(M, \mathbb{Z}_2)$ (see Subsection \ref{spinandframing}).
\item[(b)] If $M$ can be embedded in a rational homology sphere, then $K$ has a preferred rational framing (depending on the embedding). In particular, the framing cannot be changed by an ambient isotopy of $M$. Notice that this does not apply if $M$ has a properly embedded non-separating closed surface, since then it cannot be properly embedded in any rational homology sphere. \\
\end{enumerate}
\end{remark}

\subsection{Spin structures and framings}\label{spinandframing}

Since spin structures are invariants of $3$-manifolds, the framing of a knot ${\pmb K}$ in $M$ cannot be changed by one full twist using ambient isotopy. We give a short explanation of this fact below.

Let $M$ be a $3$-manifold, $TM$ its tangent bundle and $V: M \longrightarrow TM$ a vector field.
Assume that $V$ is nondegenerate.

\begin{definition}

$M$ is parallelizable if  the tangent bundle of $M$ is trivial, that is, there are three vector 
	fields $V_1$, $V_2$, and $V_3$ which form a basis at every tangent space.

\end{definition}

\begin{theorem}\cite{parallelizable} Every orientable $3$-manifold is parallelizable. 

\end{theorem}

Homotopy classes of parallelizations can be identified with spin structures and spin structures form an affine space over 
the $\mathbb{Z}_2$-linear space $H^1(M, \mathbb{Z}_2)$. The choice of a concrete parallelization makes the affine space of spin structures the linear space $H^1(M,\mathbb{Z}_2)$. In particular to every framed knot ${\pmb K} \subset M$, which represents an element of $H_1(M,\mathbb{Z}_2)$, we associate an element of $\mathbb{Z}_2$. 

Concretely, let $(V_1, V_2, V_3)$ be a fixed orthonormal parallelization of $M$, that is $(V_1, V_2, V_3)$ are orthonormal at every tangent space. 
Let ${\pmb K}$ be a framed knot in $M$ ($|K| \in H_1(M, \mathbb{Z}_2)$). This implies that $|K| \in Z_1(M, \mathbb{Z}_2)$, 
where $Z_1(M) = ker(\partial _1)$. We show that the triple $(V_1, V_2, V_3)$ defines a map from $H_1(M, \mathbb{Z}_2)$ 
to $\mathbb{Z}_2$ and this map is an element of $H^1(M, \mathbb{Z}_2)$. To see this, consider vectors $v_i \in V_i, i = 1, 2, 3$ 
which are incidental at a point $m_a \in M$ and another set of orthonormal vectors $(v_T, v_{fr}, v_{3'}) \in (V_1, V_2, V_3)$ 
which are incidental at a point $m_K \in K \subset M$. $v_T$  denotes the vector which gives the direction of travel  along 
the knot (orientation) and $v_{fr}$ is the vector which gives the framing of the knot. Therefore, there is an an element $g \in SO(3)$ 
which maps $(v_1, v_2, v_3)$ to $(v_T, v_{fr}, v_{3'})$. Now when we travel along the knot we obtain an element of 
$\pi_1(SO(3)) = \mathbb{Z}_2$. Therefore, we obtain a map from $H_1(M, \mathbb{Z}_2)$ to $\mathbb{Z}_2$. 
See \cite{ssandspin} and \cite{spinstructure} for more details.

\section{Ramifications and Connections to Skein Modules}

Our main results can be formulated in the language of skein modules as follows.  

\begin{definition}

Let $M$ be an oriented $3$-manifold and  $\mathcal{L}^{fr}$ the set of unoriented framed links in $M$ up to ambient isotopy. 
	Let $S^{fr}$ be the submodule of the module $\mathbb{Z}[q^{\pm1}]\mathcal{L}^{fr}$ generated by the framing 
	expressions $\pmb{L}^{(1)} - q\pmb{L}$ for any framed link $\pmb{L}$ in $\mathcal{L}^{fr}$. Here $\pmb{L}^{(1)}$ denotes the link obtained 
	from $\pmb{L}$ by twisting the framing of $\pmb{L}$ by a positive full twist (see Figure \ref{fig:1b}). The \textbf{framing skein module} of $M$ is defined 
	as the quotient: $$\mathcal{S}_0(M,q) = \mathbb{Z}[q^{\pm1}]\mathcal{L}^{fr} / S^{fr}. $$ 

\end{definition}
\ \ 

The following theorem for skein modules is equivalent to Corollary \ref{mainresult} as long as we work with knots (see \cite{chapterix}).

\begin{theorem}\label{mainresultskein}

For an oriented $3$-manifold $M$, 
${\mathcal S}_0(M,q)=
\mathbb{Z}[q^{\pm 1}]{\mathcal L}^f \oplus \bigoplus \limits_
{L \in ({\mathcal L^{fr}}\setminus{\mathcal L}^f)} \cfrac{\mathbb{Z}[q]}{q^{2}-1},$
where ${\mathcal L}^f$ is composed of links which do not intersect any
$2$-sphere in $M$ transversely at exactly one point.
\end{theorem}

\begin{proof}

In Theorem \ref{Link} we described two possibilities of changing the framing of ${\pmb L}$ by an ambient isotopy. We analyze both cases for the proof.

\begin{enumerate}

    \item If $L$ intersects a non-separating sphere $S^2_i$ transversely in exactly one point, say by component $K_i$, then $\tau_i^2$ is ambient isotopic to the identity and twists the framing
        of ${\pmb K_i}$ by two full twists, and thus also the framing of ${\pmb L}$. Therefore, in the framing skein module $(q^2-1){\pmb L}=0$. We notice that
        $\mathbb{Z}[q^{\pm 1}]\mathcal{L}^{fr}$ divided by this relation exactly gives $\mathbb{Z}[q^{\pm 1}]{\mathcal L}^f \oplus \bigoplus \limits_
{L \in ({\mathcal L^{fr}}\setminus{\mathcal L}^f)} \frac{\mathbb{Z}[q]}{q^{2}-1}$. \\ 
        
    \item If exactly two components $K_i$ and $K_j$ of the link $L$ intersect a non-separating $2$-sphere
        $S^2_{i,j}$ in one point each, then $\tau_{i,j}^2$ changes the framing of ${\pmb K_i}$ by $2$ and of ${\pmb K_j}$ by $-2$ (as illustrated in Figure \ref{linklbl}).  Thus, even though this Dehn homeomorphism changes the framing of ${\pmb L}$, it is invisible in the framing skein module
        which does not see which component is twisted. Therefore, the twists cancel algebraically in ${\mathcal S}_0(M,q)$.
    
\end{enumerate}

\end{proof}

\begin{corollary}\label{corframingsm}
There exists an epimorphism $\omega$ from the framing skein module to the $\mathbb{Z}[q]/(q^2-1)$-group ring 
	over the first homology with $\mathbb{Z}_2$ coefficients, that is
	$$\omega: \mathcal{S}_0(M,q) \longrightarrow \bigg(\frac{\mathbb{Z}[q]}{q^2-1}\bigg)H_1(M,\mathbb{Z}_2),$$  
which is not canonical and depends on the choice of spin structure (in the form of a parallelization of a tangent bundle to $M$). The choice of parallelization gives
a map $b:H_1(M,\mathbb{Z}_2) \longrightarrow \mathbb{Z}_2$. Hence, $\omega(K)= |K|$ if $b(|K|)=0$ and $\omega(K)= q|K|$ if $b(|K|)=1$.
Here ${\pmb K}$ is a framed knot in $M$ and $|K|$ denotes the homology class of $K$ in $H_1(M, \mathbb{Z}_2)$. See Subsection \ref{spinandframing}.

\end{corollary}

\subsection{From the Kauffman bracket skein module to spin twisted homology}

\begin{definition}\cite{smof3} Let $M$ be an oriented $3$-manifold, $\mathcal{L}^{fr}$ the set of unoriented framed links (including 
the empty link $\varnothing$) in $M$ up to ambient isotopy, $R$ a commutative ring with unity, and $A$ 
a fixed invertible element in $R$.
In addition, let $R\mathcal{L}^{fr}$ be the free $R$-module generated by $\mathcal{L}^{fr}$ and $S_{2, \infty}^{sub}$ the 
	submodule of $R\mathcal{L}^{fr}$ generated by all (local) skein expressions of the form:

\begin{itemize}
    \item [(i)]$\pmb{L}_+ - A\pmb{L}_0 - A^{-1}\pmb{L}_{\infty}$, and
    \item [(ii)] $\pmb{L} \sqcup \pmb{\bigcirc}  + (A^2 + A^{-2})\pmb{L}$,
\end{itemize}

where $\pmb{\bigcirc}$ denotes the trivial framed knot and the {\it skein triple} $\pmb L_+$, $\pmb L_0$ and $\pmb L_{\infty}$ denote three 
	framed links in $M$ which are identical except in a small $3$-ball in $M$ where they differ as shown below.

\[  \begin{minipage}{1.3in} \includegraphics[width=\textwidth]{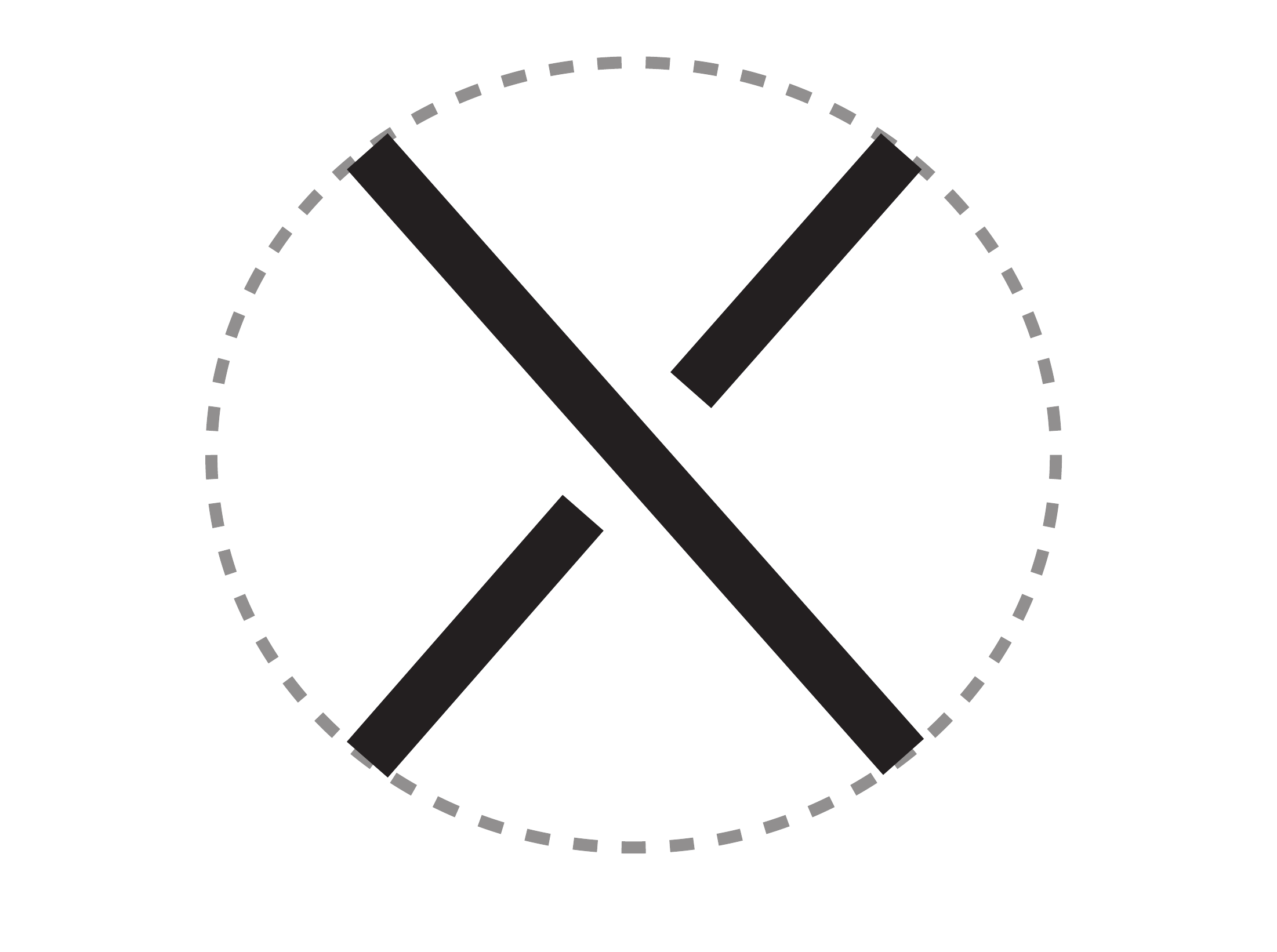} \vspace{-15pt} \[\pmb L_+\] \end{minipage} 
               \qquad
        \begin{minipage}{1.3in}\includegraphics[width=\textwidth]{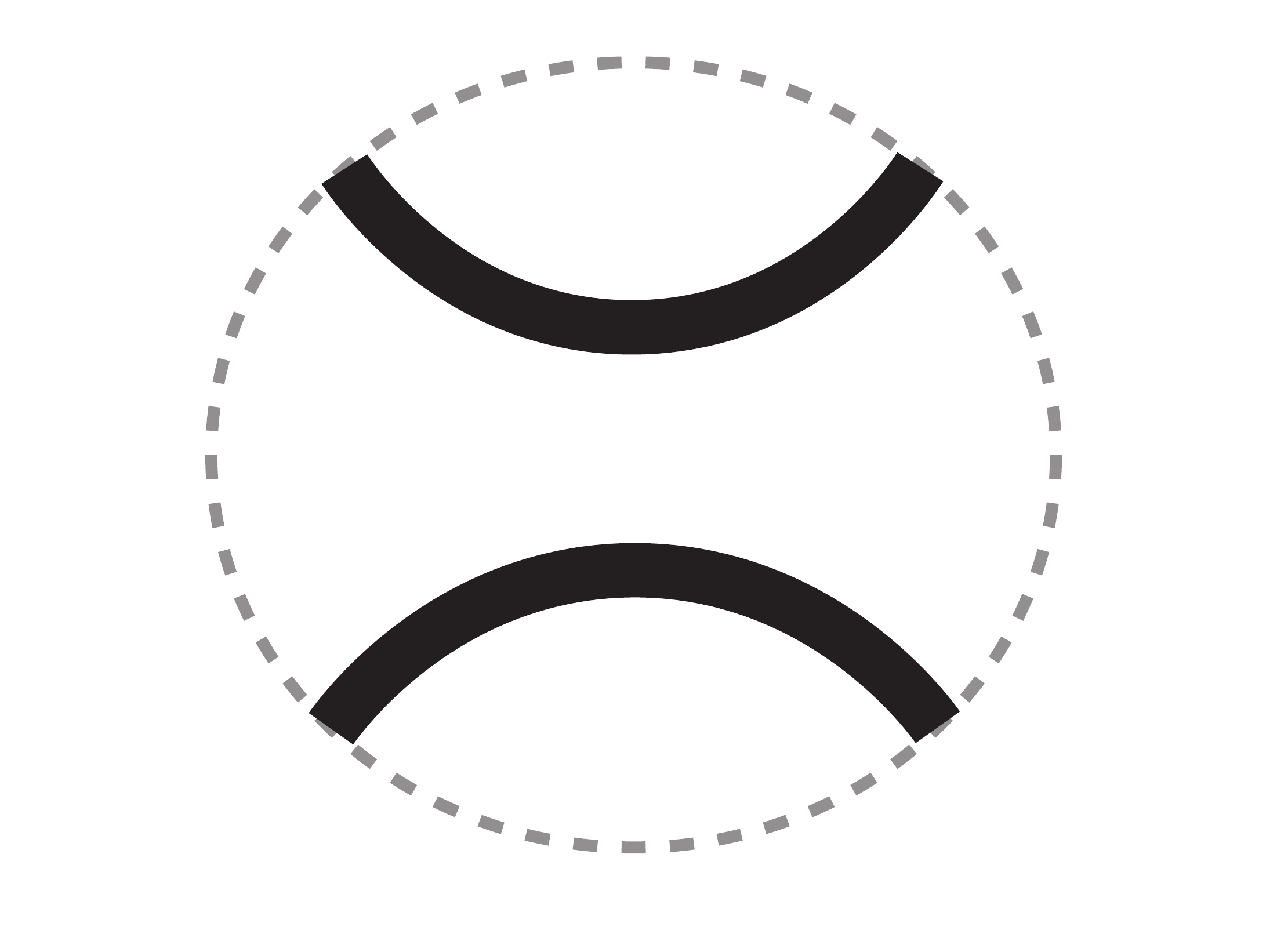} \vspace{-15pt} \[\pmb L_0\] \end{minipage}
         \qquad
        \begin{minipage}{1.3in}\includegraphics[width=\textwidth]{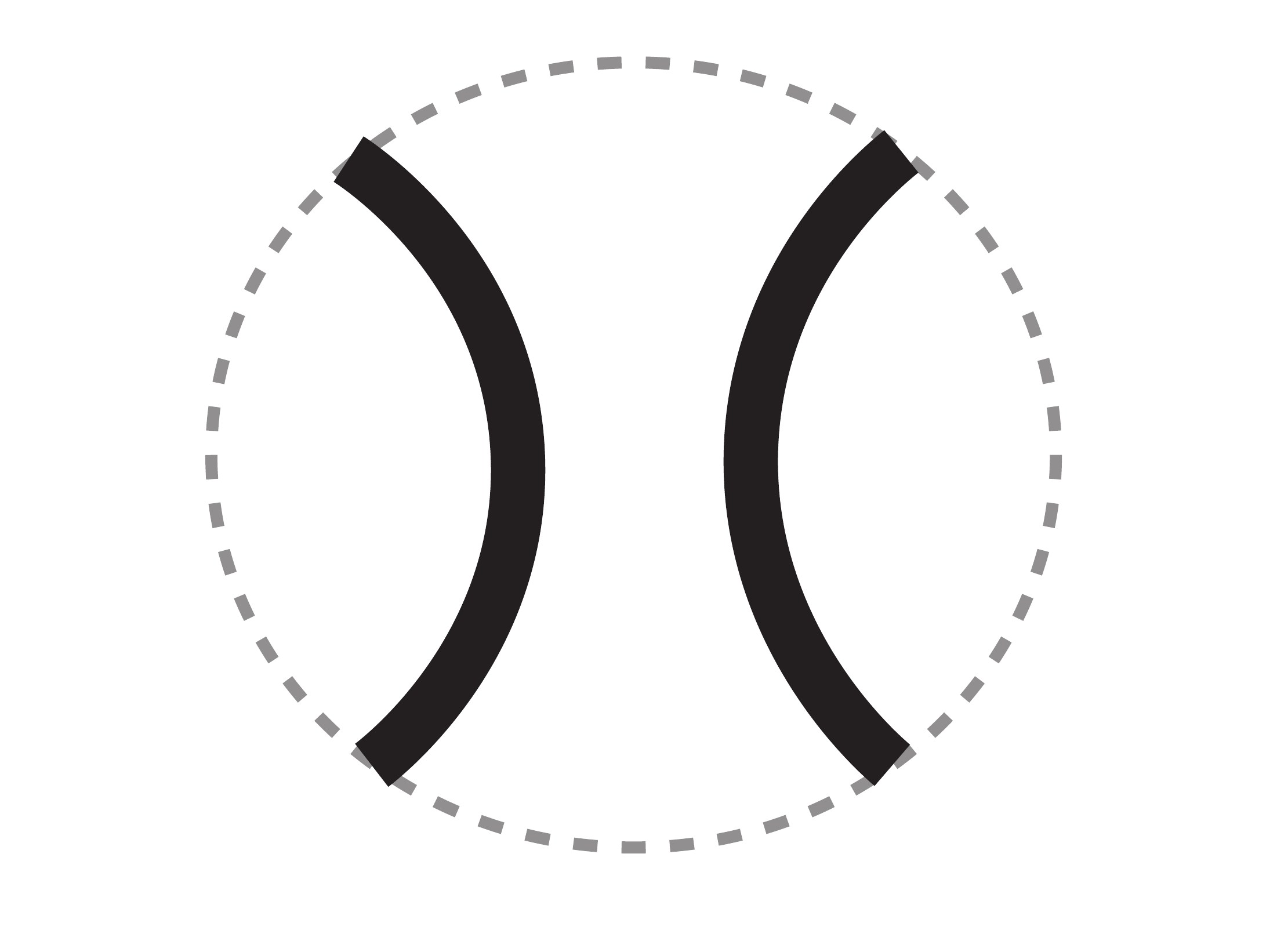} \vspace{-15pt} \[ \pmb L_\infty\]\end{minipage} 
        \]

The \textbf{Kauffman bracket skein module} (KBSM) of $M$ is defined as the quotient: 
	$$\mathcal{S}_{2,\infty}(M;R,A) = R\mathcal{L}^{fr}/S_{2, \infty}^{sub}.$$

\end{definition}

Notice that $\pmb{L}^{(1)} = -A^3\pmb{L}$ in $\mathcal{S}_{2,\infty}(M;R,A)$. This relation is called the {\it framing relation}. 
For simplicity we will use the notation $\mathcal{S}_{2,\infty}(M)$ when $R = \mathbb{Z}[A^{\pm 1}]$.\\ 

\begin{proposition}

There exists an epimorphism
$$\phi : \mathcal{S}_{2,\infty}(M) \longrightarrow \bigg(\frac{\mathbb{Z}[A]}{A^4+A^2+1}\bigg) H_1(M, \mathbb{Z}_2)$$  
	which is not canonical and depends on the choice
of spin structure on $M$ (here in the form of a parallelization of a tangent bundle to $M$). The codomain of $\phi$ is called \textbf{spin twisted homology}. Compare with Corollary \ref{corframingsm} and \cite{ssandspin,ps1}. 

\end{proposition}

\begin{proof}

Using the parallelization of the tangent bundle of $M$, we have a map \\ $\widetilde{\phi} : \mathbb{Z}[A^{\pm 1}]\mathcal L^{fr} \longrightarrow  \bigg(\frac{\mathbb{Z}[A^{\pm1}]}{A^4+A^2+1}\bigg) H_1(M, \mathbb{Z}_2)$. We check that  $\widetilde{\phi}$ is zero on the Kauffman bracket skein expressions:

\begin{enumerate}

\item Since $\widetilde{\phi}(\pmb{L}_+) = -A^{-3}\widetilde{\phi}(\pmb{L}_0)= -A^{-3}\widetilde{\phi}(\pmb{L}_{\infty})$, then $\widetilde{\phi}(\pmb{L}_+ - A\pmb{L}_0 - A^{-1}\pmb{L}_{\infty}) = 0$. 

\item Since $\pmb{L} \sqcup \pmb{\bigcirc}  = (-A^2 - A^{-2})\pmb{L}$ in $\mathcal{S}_{2,\infty}(M)$, then $(A^4 + A^2 +1)\pmb{L} = 0$ in $H_1(M, \mathbb{Z}_2)$.

\end{enumerate}

\end{proof}

\subsection{The $q$-homology skein module}

\begin{definition}

Let $M$ be an oriented $3$-manifold, $\mathcal{L}^{\overrightarrow{fr}}$ denote the set of ambient isotopy classes of oriented
framed links in $M$ and let $R = \mathbb{Z}[q^{\pm1}]$.
Let us denote by $S_2^{\overrightarrow{fr}}$ the submodule of $R\mathcal{L}^{\overrightarrow{fr}}$ generated by the skein expressions $\pmb{L}_+ - q\pmb{L}_0$ and $\pmb{L}^{(1)} - q\pmb{L}$,
illustrated in Figure \ref{qhomologyfigure}. 

The quotient $R\mathcal{L}^{\overrightarrow{fr}}/S_2^{\overrightarrow{fr}}$ is called the $\pmb q$\textbf{-homology skein module}, also known as the second skein module, and it is denoted
by $\mathcal{S}_2(M, q)$.

\end{definition}

\begin{figure}[h]
\centering
\begin{subfigure}{.49\textwidth}
\centering
\begin{minipage}{1.3in}\includegraphics[width=\textwidth]{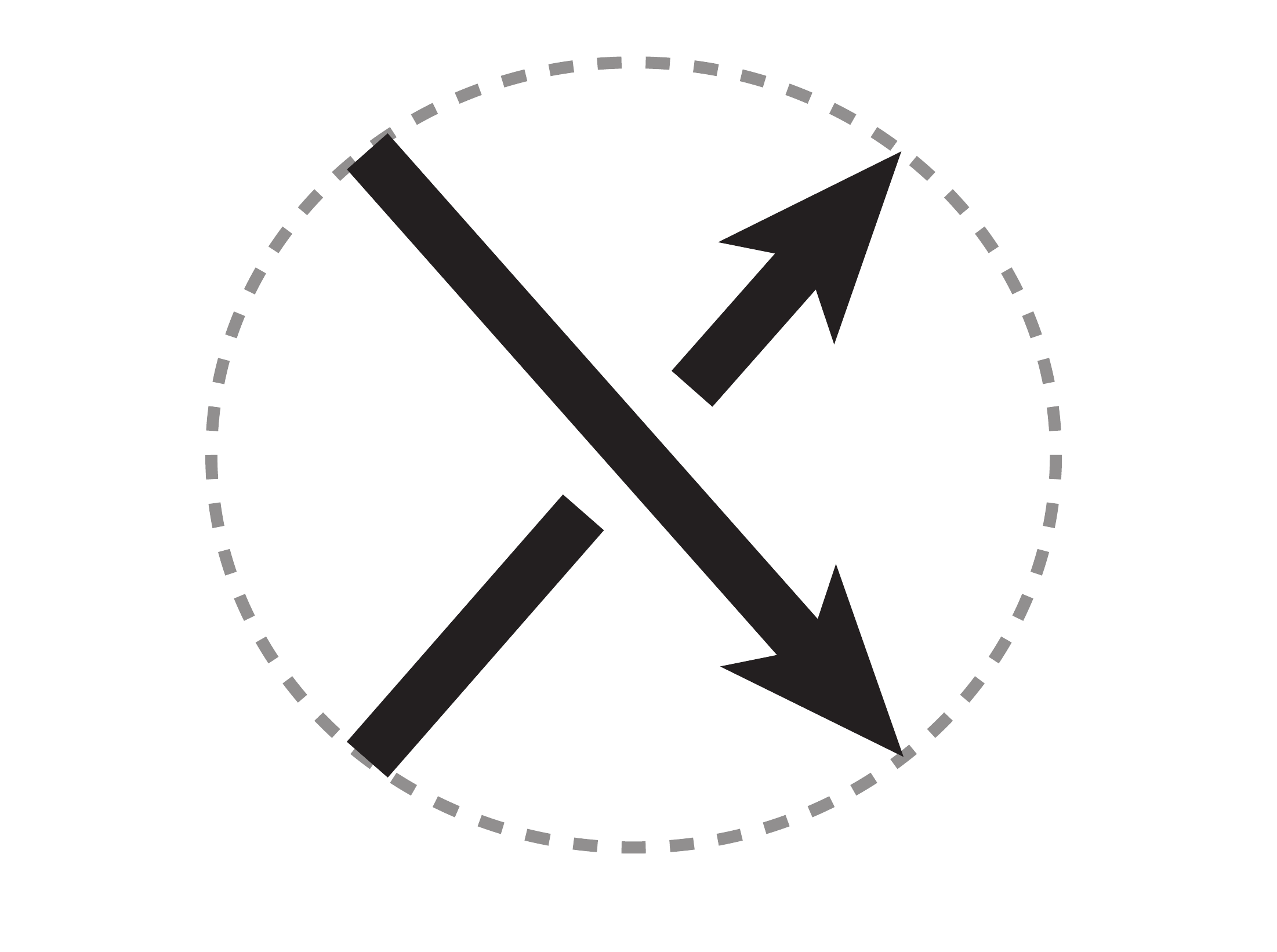}\vspace{-10pt} \[\pmb L_+\]\end{minipage}
- q \begin{minipage}{1.3in}\includegraphics[width=\textwidth]{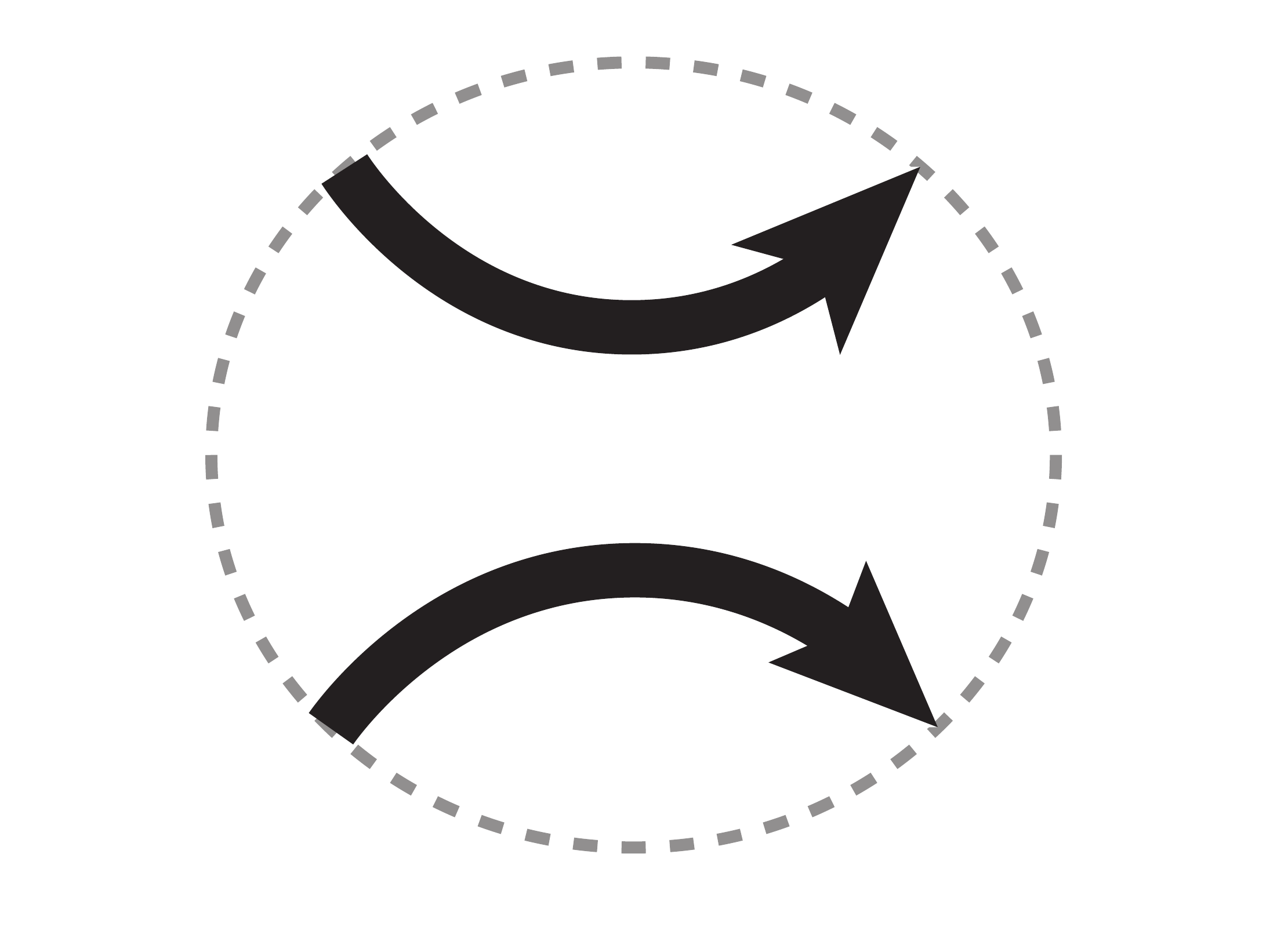}\vspace{-10pt} \[\pmb L_0\]\end{minipage} 
\subcaption{q-homology skein relation}
\end{subfigure}
\begin{subfigure}{.49\textwidth}
\begin{minipage}{1.3in}\includegraphics[width=\textwidth]{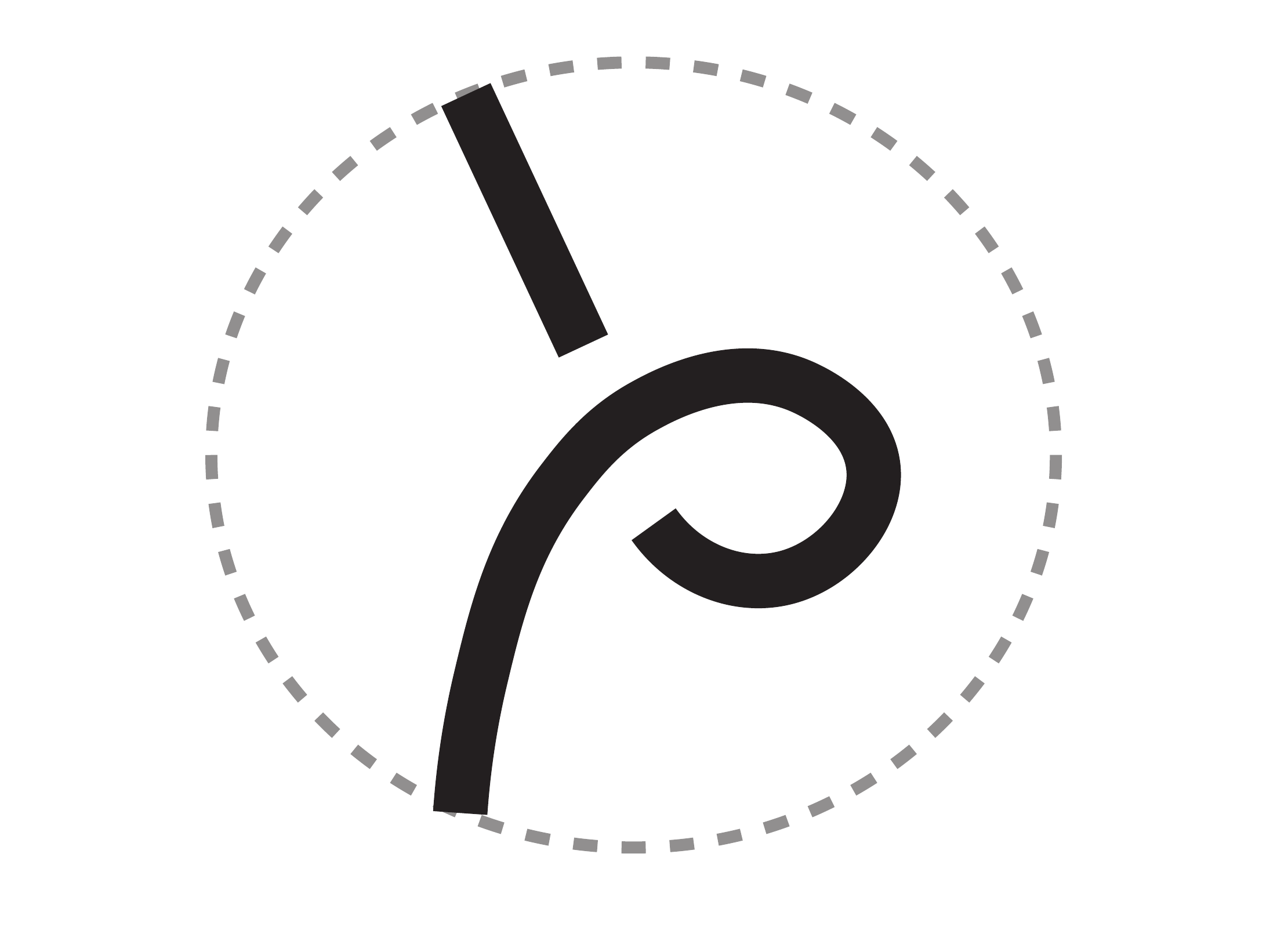}\vspace{-10pt} \[\pmb L^{(1)}\]\end{minipage}
- q \begin{minipage}{1.3in}\includegraphics[width=\textwidth]{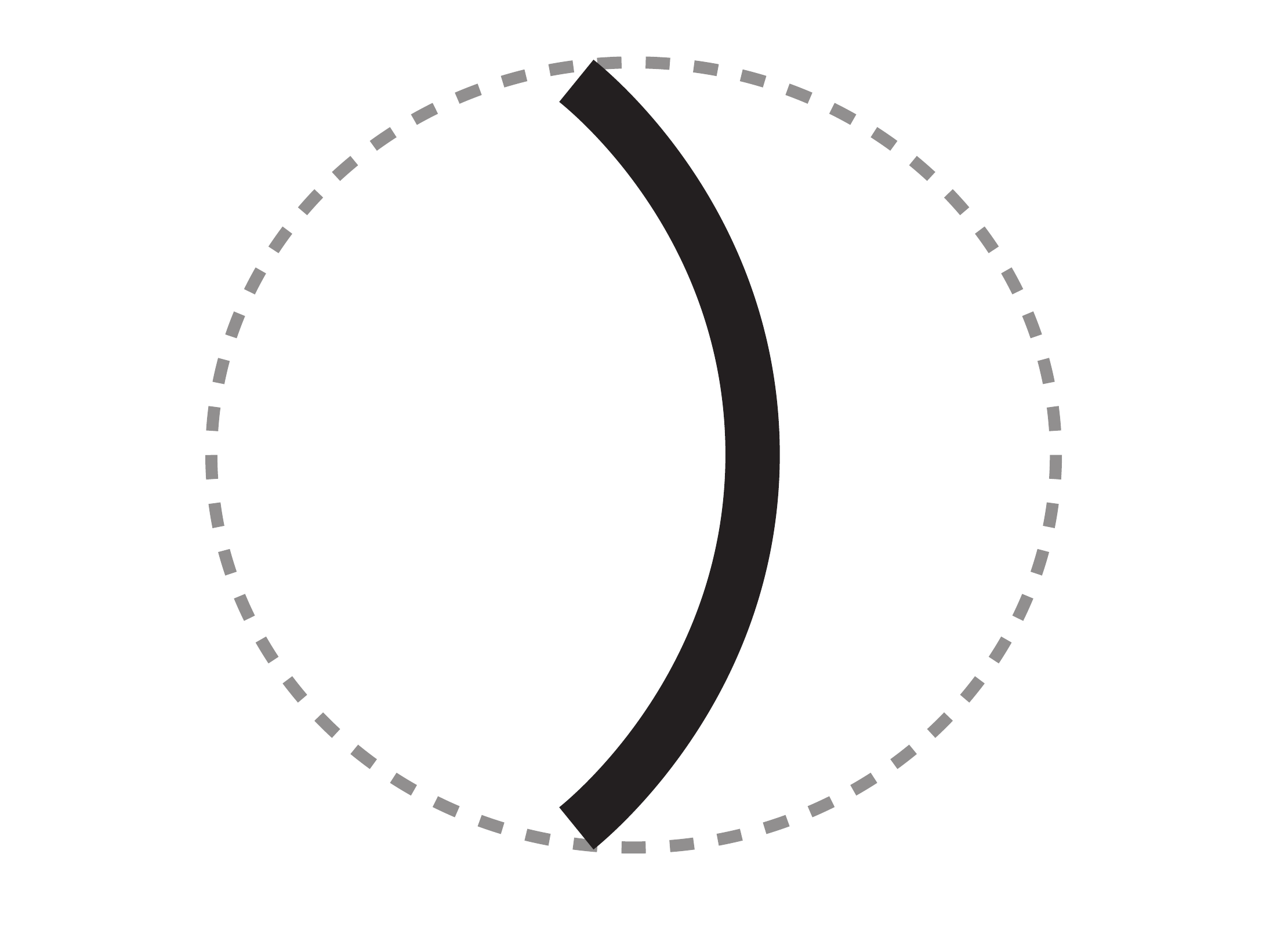}\vspace{-10pt} \[\pmb L\]\end{minipage}
\subcaption{q-homology framing relation}\label{fig:1b}
\end{subfigure}
\caption{\color{white}.}
\label{qhomologyfigure}
\end{figure}

\begin{proposition}

There exists an epimorphism $$\psi : \mathcal{S}_2(M, q) \longrightarrow \bigg(\frac{\mathbb{Z}[q]}{q^{2}- 1}\bigg) H_1(M, \mathbb{Z}_2),$$

which is not canonical and depends on the choice
of spin structure on $M$. As before, compare with Corollary \ref{corframingsm}. 

\end{proposition}

We summarize our propositions with the following diagram:

\tikzcdset{arrow style=tikz,
diagrams={>={Straight Barb[scale=1]}} }

\[
\begin{tikzcd}[row sep=tiny]
\mathbb{Z}[q^{\pm1}]\mathcal{L}^{\overrightarrow{fr}}\arrow[r] & \mathcal{S}_2(M, q) \arrow[r]\arrow[drr, bend left = 40, "\psi", two heads] & \frac{\mathcal{S}_2(M, q)}{q^2 -1} \arrow[rd, "spin"] &  & &\\ 
& & &  \bigg(\frac{\mathbb{Z}[q]}{q^{2}- 1}\bigg) H_1(M, \mathbb{Z}_2) \arrow[dd, "q\to-A^3"]\arrow[rd] & & \\
\mathbb{Z}[q^{\pm1}]\mathcal{L}^{fr} \arrow[r] & \mathcal{S}_0(M, q) \arrow[r]\arrow[rru, bend left = 17, "\omega",two heads] & \frac{\mathcal{S}_0(M, q)}{q^2 -1} \arrow[ru, "spin"] \arrow[d, "q\to-A^3"]  &  & H^1(M, \mathbb{Z}_2) \\[35pt]
\mathbb{Z}[A^{\pm 1}]\mathcal{L}^{fr} \arrow[r] & \mathcal{S}_{2, \infty}(M) \arrow[r] \arrow[rr, bend right = 30, "\phi", two heads] & \frac{\mathcal{S}_{2, \infty}(M)}{A^4 + A^2 + 1} \arrow[r, "spin"] & \bigg(\frac{\mathbb{Z}[A]}{A^4+A^2+1}\bigg)H_1(M, \mathbb{Z}_2) \arrow[ur] &  \\
\end{tikzcd}
\]

\section{Future Directions}

As we have proven in Theorem \ref{mainresultskein}, the presence of a non-separating $S^2$ in $M$ always yields 
torsion in $\mathcal{S}_0(M,q)$. 
More generally, skein modules that have framing relations always have torsion which is obtained using the light bulb trick. Usually, however, skein modules have more torsion than what the framing relations give us. For example, in the $q$-homology skein module, torsion is related to the presence of closed non-separating surfaces \cite{qanalogue}. 
In the Kauffman bracket skein module torsion is obtained using incompressible spheres and tori (see \cite{kirbyproblemlist, ohtsukiproblemlist}). 
This leads us to the following conjecture. \\

\begin{conjecture}\label{torsionconj} \hfill

\begin{enumerate}

 \item If an irreducible, compact oriented $3$-manifold $M$ is non-Haken, then $\mathcal{S}_{2,\infty}(M)$ is torsion free (see \cite{kirbyproblemlist} and \cite{ohtsukiproblemlist}). \\

    \item If a compact oriented $3$-manifold $M$ contains an incompressible $2$-sphere or torus which is non-parallel to the boundary, then $\mathcal{S}_{2,\infty}(M)$ contains torsion (this was conjectured by the fourth author; see \cite{kirbyproblemlist}). 
    
\end{enumerate}

\end{conjecture}

The following statement was conjectured by Witten (a proof was announced in \cite{wittenresolved}).

\begin{conjecture}[Witten]\label{wittenresolved}

The Kauffman bracket skein module for any closed oriented $3$-manifold over $\mathbb{Q}(A)$, the field of rational functions in the variable $A$, is always finite dimensional. 

\end{conjecture}

We can combine this conjecture with Conjecture \ref{torsionconj} to obtain the following conjecture. 

\begin{conjecture}

$\mathcal{S}_{2, \infty}(M)$ of a closed, oriented, irreducible, non-Haken $3$-manifold has finite rank. 

\end{conjecture}

\section{Acknowledgements}

The fourth author was partially supported by Simons Collaboration Grant-637794 and the CCAS Dean’s Research Chair award.

\end{document}